\documentclass[a4paper,12pt]{article}

\usepackage{amsfonts}
\usepackage{amsmath}
\usepackage{amsbsy}
\usepackage{amsxtra}
\usepackage{latexsym}
\usepackage{amssymb}
\usepackage[active]{srcltx}

\usepackage{algorithm, algorithmic}
\usepackage{color,fancybox,graphicx}
\DeclareMathOperator{\cone}{cone}

\DeclareMathOperator{\RSS}{RSS}

\DeclareMathOperator{\trace}{tr}

\newcommand{\E}{\mathbb E}
\newcommand{\iso}{\mathrm{iso}}
\newcommand{\anti}{\mathrm{anti}}
\newcommand{\nv}{\mathrm{V}}

\begin{document}

\begin{center}

{\sc \Large A Geometrical Approach to Iterative Isotone Regression \\
\vspace{0.2cm}
}

\vspace{0.5cm}

Arnaud GUYADER\footnote{Corresponding author.}\\
Universit\'e Rennes 2, INRIA and IRMAR \\
Campus de Villejean, Rennes, France\\
\textsf{arnaud.guyader@uhb.fr}
\vspace{0.5cm}

 Nicolas J\'EGOU\\
 Universit\'e Rennes 2\\
Campus de Villejean, Rennes, France\\
\textsf{nicolas.jegou@uhb.fr}
\vspace{0.5cm}

 Alexander B. N\'EMETH\\
Faculty of Mathematics and Computer Science\\
Babe\c s Bolyai University, RO-400084 Cluj-Napoca, Romania\\
\textsf{nemab@math.ubbcluj.ro}
\vspace{0.5cm}

 S\'andor Z. N\'EMETH\\
School of Mathematics, The University of Birmingham\\
Birmingham B15 2TT, United Kingdom\\
\textsf{nemeths@for.mat.bham.ac.uk}

\end{center}

\begin{abstract}
\noindent In the present paper, we propose and analyze a novel method for estimating a univariate regression function of bounded variation. The underpinning idea is to combine two classical tools in nonparametric statistics, namely isotonic regression and the estimation of additive models. A geometrical interpretation enables us to link this iterative method with Von Neumann's algorithm. Moreover, making a connection with the general property of isotonicity of projection onto convex cones, we derive another equivalent algorithm and go further in the analysis. As iterating the algorithm leads to overfitting, several practical stopping criteria are also presented and discussed.
\medskip

\medskip

\end{abstract}

\newenvironment{proof}{{\bf Proof.}}{\hfill$\Box$\\}
\newtheorem{theorem}{Theorem}
\newtheorem{definition}{Definition}
\newtheorem{lemma}{Lemma}
\newtheorem{proposition}{Proposition}
\newtheorem{example}{Example}
\newtheorem{corollary}{Corollary}

\newcommand{\defi}{{\bf  \hspace*{2 mm} Definition\ }}
\newcommand{\arr}{\longrightarrow}
\newcommand{\Arr}{\Rightarrow}
\newcommand{\R}{\mathbb R}
\newcommand{\N}{\mathbb N}
\newcommand{\al}{\alpha}
\newcommand{\la}{\lambda}
\newcommand{\lang}{\langle}

\newcommand{\rang}{\rangle}

\newcommand{\conv}{\textrm{co}}

\newcommand{\stm}{\setminus}
\newcommand{\sbs}{\subset}
\newcommand{\mc}{\mathcal}
\newcommand{\ri}{\textrm{ri}}
\newtheorem{theo}{Theorem}
\newtheorem{lem}{Lemma}
\newtheorem{cor}{Corollary}
\newtheorem{prp}{Proposition}
\newcommand{\edem}{$\hspace{\stretch{22}}\Box$\vskip3mm}

\section{Introduction}\label{sec:introduction}
In a statistical setting, consider the nonparametric regression model 
\begin{equation}\label{eq:general-model}
Y=r(X)+\varepsilon
\end{equation}
where $X$ and $Y$ are both real-valued random variables with $X$ uniform in $[0,1]$, $\E\left[Y^2\right]<\infty$ and $\E\left[\varepsilon|X\right]=0$ (see for example \cite{Laciregressionbook}). Assume, in addition, that the regression function $r$ is right-continuous and of bounded variation. With this respect, the Jordan decomposition asserts that such a function can be written as the sum of a non-decreasing function $u$ and a non-increasing function $b$ 
\begin{equation*}
r(x)=u(x)+b(x).
\end{equation*}
Viewing this latter equation as an additive model involving the increasing part and the decreasing part of $r$, we propose a new estimator which combines two well-established tools in nonparametric regression: the isotonic regression related to the estimation of monotone functions, and the backfitting algorithm devoted to the estimation of additive models.

\medskip

The estimation of a monotone regression function dates back to the 50's and the early work by Ayer et al. \cite{ayer1955empirical}. Given a sample of independent and identically distributed (i.i.d.) random couples $(X_1,Y_1),\ldots,(X_n,Y_n)$ following the general model (\ref{eq:general-model}), denoting  $x_1=X_{(1)}<\ldots< x_n=X_{(n)}$ the ordered sample, and $y_1,\ldots, y_n$ the corresponding observations, the Pool-Adjacent-Violators Algorithm (PAVA) determines a collection of non-decreasing level sets solution to the minimization problem
\begin{equation}\label{eq:min1}
\min_{u_1\leq \ldots \leq u_n} \sum_{i=1}^n \left(y_i-u_i\right)^2.
\end{equation}
Since the cone  
\begin{equation}\label{eq:isotone-cone}
{\cal C}^+=\left\{u=\left(u_1,\ldots,u_n\right)\in \R^n:u_1\leq \ldots\leq u_n\right\}
\end{equation}
is a closed convex set in $\R^n$, there exists a unique solution to this minimization problem. This solution, called the isotonic regression of $y$ and denoted $\iso(y)$, is the metric projection of $y=(y_1,\ldots,y_n)$ on ${\cal C}^+$ with respect to the Euclidean norm, that is
\begin{equation}\label{eq:iso-y}
\iso(y)=\arg\min_{u\in {\cal C}^+}\|y-u\|^2=\arg\min_{u\in {\cal C}^+}\sum_{i=1}^n \left(y_i-u_i\right)^2.
\end{equation}
Correspondingly, the antitonic regression of $y$ is the projection of $y$ on the set of vectors with non-increasing coordinates, that is
\begin{equation*}
\anti(y)=\arg\min_{b\in {\cal C}^-}\|y-b\|^2=\arg\min_{b\in {\cal C}^-}\sum_{i=1}^n \left(y_i-b_i\right)^2
\end{equation*}
where ${\cal C}^-=-{\cal C}^+=\left\{b=\left(b_1,\ldots,b_n\right)\in \R^n:b_1\geq \ldots\geq b_n\right\}$. From now on, ${\cal C}^+$ and ${\cal C}^-$ will be called monotone cones.

\medskip
 
A major attraction of isotonic regression procedures is their simplicity. Since they are nonparametric and data driven (i.e., they do not require the tuning of any smoothing parameter), these estimators have raised considerable interest since more than fifty years. A comprehensive account on the subject can be found in \cite{barlow1972statistical}, statistical properties have been studied in \cite{brunk1955maximum}, \cite{brunk1970estimation}, \cite{wright1981asymptotic}, \cite{durot2007error}, and extensions or improvements to more general order of the PAVA approach can be found in \cite{dykstra1981isotonic}, \cite{lee1983min} and \cite{BestChakravarti1990} for example.

\medskip

Still in nonparametric statistics, but in a multidimensional context this time, the additive models were suggested by Friedman and Stuetzle \cite{friedman1981projection} and popularized by Hastie and Tibshirani \cite{hastie1990generalized} as a way to get around the so-called ``curse of dimensionality''. In brief, this means that, in multivariate smoothing, nonparametric estimators have to consider large neighborhoods of a particular point of the space to catch observations, and hence large biases can result. The additive model assumes that the regression function can be written as the sum of smooth terms in the covariates: 
\begin{equation}\label{eq:additive-models}
r(X)=\sum_{j=1}^dr_j(X^{j}).
\end{equation}

Since each variable is represented separately in (\ref{eq:additive-models}), the additive model provides a logical extension of the standard linear regression and once an additive model is fitted, one can easily interpret the role of each variable in predicting the response.

\medskip

 Buja et al. \cite{buja1989linear} proposed the backfitting algorithm as a practical method for fitting additive models. It consists in iterated fitting of the \textit{partial residuals} from earlier steps until convergence is reached. If the current estimates are $\hat{r}_1,\ldots, \hat{r}_d$, then $\hat{r}_j$ is updated by smoothing $y-\sum_{k\neq j} \hat{r}_k$ against $X^{j}$. While backfitting has attracted much attention and is frequently applied, it has been somewhat difficult to analyze theoretically. Nonetheless, when using linear smoothers in each direction, the convergence of the algorithm can be related to the spectrum of the individual smoothing matrices (see, e.g., \cite{buja1989linear} and \cite{opsomer1997fitting}), and when all the smoothers are orthogonal projections, the whole algorithm can be replaced by a global projection operator \cite{hardle1993backfitting}. 

\medskip

There exist other multivariate methods based on repeated fitting of the residuals. Some of them, like L2-boosting \cite{buhlmann2003boosting}, boosted kernel regression \cite{di2008boosting}, iterative bias reduction \cite{cornillon2011,cornillon2012iterative}, do not assume any particular structure for the regression function. The common principle of these approaches is to start out with a biased smoother or a weak learner, and then to estimate and correct the bias in an iterative manner. Hence, instead of smoothing the \textit{partial residuals}, one smoothes the \textit{global residuals} $y-\sum_k \hat{r}_k$, and then correct the previous smoother. Just as for the backfitting, the convergence of these algorithms as well as the statistical properties of these estimators have mainly been studied in the case of linear smoothers.   

\medskip

In our situation, however, it is noteworthy that projections on convex cones are not linear operators. But considering our iterative estimator as the application of Von Neumann's algorithm (see for example \cite{Dattorro2005}), we will show that iterating the procedure tends to reproduce the data. Moreover, we manage to go further in the analysis by proving that the individual terms of the sum converge as well to identified limits. This is in fact possible thanks to a result which is rather unexpected from the statistical point of view: iterating isotonic regression in a backfitting fashion or in a boosting fashion yield the same estimators at each step.

\medskip

Interestingly, this result stems from a  property of the projections onto monotone cones which, in our case, reads as follows (recall that $\iso(y)$ is defined by (\ref{eq:iso-y})):
\begin{equation}\label{eq:isotonicity} 
\forall (y,\tilde{y})\in\R^n\times\R^n,\qquad y-\tilde{y}\in{\cal C}^+\ \Rightarrow\ \iso(y)-\iso(\tilde{y})\in {\cal C}^+
\end{equation}
From a more general perspective, one can see this equation as a particular case of the property of isotonicity of the projection onto convex cones. Here isotonicity is considered with respect to the order induced by the cone. The idea to relate the ordering induced by a convex cone and the metric projection onto the convex cone goes back to the paper by Isac and N\'emeth \cite{isac1986monotonicity}, where a convex cone in the Euclidean space which admits an isotone projection onto it (called by the authors \emph{isotone projection cone}) was characterized. Thereafter, this notion was considered in the complementarity theory to provide new existence results and iterative methods \cite{IsacNemeth1990b,IsacNemeth1990c,Nemeth2009a}.

\medskip

Yet, the notion of the cone in the above cited papers is used in the sense of ``closed convex pointed cone''. Confronted with the question if the monotone cones ${\cal C}^-$ and ${\cal C}^+$, which are not pointed, admit or not isotonic metric projections, we shall develop in Section \ref{sec:isotonicity-projection} a general theory in order to apply it to this special case. This seems to be the simplest way to tackle this problem. Therefore, Theorem \ref{fo} below is interesting by itself. By using this approach, Corollary \ref{coriso} states that the monotone cones ${\cal C}^-$ and ${\cal C}^+$ admit isotone metric projections.

\medskip

Then, we come back to the statistical framework, the remainder of the paper being organized as follows: the definitions, the analysis and the equivalence of the Iterative Isotone Regression estimator and the Iterative Isotone Bias Reduction estimator are detailed in Section \ref{sec:iir}. 
As iterating these algorithms tends to reproduce the data, we then explain how the procedure might be stopped in practice (Section \ref{sec:simulation}). Finally, most of the proofs are gathered in the Appendix. 

\medskip

To conclude this introduction, we would like to make a few comments on the topics that will {\bf not} be addressed in the present document. Starting from a sample ${\cal D}_n=\{(X_1,Y_1),\dots,(X_n,Y_n)\}$ of i.i.d. random couples with the same distribution as a generic pair $(X,Y)$, our estimator takes the form $\hat{r}_n^{(k_n)}$, where $k_n$ denotes the number of iterations possibly depending on the sample size $n$. In this framework, an important aspect is to specify conditions on the regression model (\ref{eq:general-model}) and on the sequence $(k_n)$ so that, for example, 
$$\E\left[\left(r(X)-\hat{r}_n^{(k_n)}(X)\right)^2\right]\xrightarrow[n\to\infty]{}0,$$
where the expectation $\E[.]$ is considered with respect to the sample ${\cal D}_n$ and the generic variable $X$. If this property is satisfied, we say that the regression function $\hat{r}_n^{(k_n)}$ is consistent. It is universally consistent if this property is true for all distributions of $(X,Y)$ with $\E[Y^2]<\infty$ (see for example \cite{Laciregressionbook}). This important issue will not be pursued further here and will be addressed elsewhere by the authors.

\section{Isotonicity of the projection onto the monotone cones}\label{sec:isotonicity-projection}
It turns out that the isotonicity of the projection, as defined in (\ref{eq:isotonicity}) for the specific case of the cone ${\cal C}^+$, is in general a very strong requirement which implies the latticiality of the order induced by the convex cone. Thus, the investigation of the isotone projection cones becomes part of the theory of latticially ordered Euclidean and Hilbert spaces. A simple finite method of projection onto isotone projection cones proposed in \cite{NemethNemeth2009} has become important in the effective handling of all the problems involving projection onto these cones. Besides nonlinear complementarity, isotone projection cones have applications in other domains of optimization theory.
The method proposed in \cite{NemethNemeth2009} has become important in the effective handling of the problem of map-making from relative distance information e.g., stellar cartography, see Section 5.13.2.4 in \cite{Dattorro2005} and
{\small
\begin{verbatim}
www.convexoptimization.com/wikimization/index.php/Projection_on_Polyhedral_Cone 
\end{verbatim}}
Although we shall not consider projection methods here, we stress that some of the results developed in \cite{NemethNemeth2009} will be useful in our proofs. 

\medskip

Let us first introduce some notations. If ${\cal C}$ is a non-empty, closed convex set in $\R^n$, then for each $x\in \R^n$ there exists a unique nearest point $P_{\cal C}x\in {\cal C}$, that is, a point with the property that
$$\|x-P_{\cal C}x\|=\inf \{\|x-c\|:\;c\in {\cal C}\},$$
where $\|.\|$ stands for the Euclidean norm in $\R^n$ \cite{Zarantonello1971}.
 The mapping $P_{\cal C}:\R^n\to {\cal C}$ is called the nearest point mapping of $\R^n$ onto ${\cal C}$ or simply the (metric) projection onto ${\cal C}$. Let ${\cal C}$ be a closed convex cone in $\R^n$, i.e., a closed nonempty set with $(i)\ {\cal C}+{\cal C}\subset {\cal C}$, and $(ii)\ t{\cal C}\subset {\cal C},\;\forall \;t\in \R_+ =[0,+\infty)$. If ${\cal C}\cap (-{\cal C})=\{0\}$, then ${\cal C}$ is called a closed convex pointed cone. In order to lighten the writings, and since all sets at hand in the following will be closed and convex, we propose to call them respectively ``cone'' and ``pointed cone''. 

\begin{lemma}\label{ossz}
Suppose that ${\cal C}$ is a cone, denote $L={\cal C}\cap (-{\cal C})$ the maximal subspace contained in ${\cal C}$, $L^\perp$
its orthogonal complement, and $K=L^\perp \cap {\cal C}$. Then, $K$ is a pointed cone in $L^\perp$,
\begin{equation}\label{ortossz}
{\cal C}=K\oplus L
\end{equation}
where $\oplus$ stands for the orthogonal sum, and
\begin{equation}\label{projek}
P_{\cal C}x=P_Kx_k+x_l
\end{equation}
where $x=x_k+x_l$ with $x_l\in L$ and $x_k\in L^\perp.$
\end{lemma}

By putting $u\leq_{\cal C} v$ whenever $u,\;v\in \R^n$ and $v-u\in {\cal C}$, the cone ${\cal C}\subset \R^n$ induces a semi-order $\leq_{\cal C}$ in $\R^n$ which is translation invariant (i.e. $u\leq_{\cal C} v$ implies $u+z\leq_{\cal C} v+z$ for any $z\in \R^n$) and scale invariant (i.e. $u\leq_{\cal C} v$ implies $tu\leq_{\cal C} tv$ for any $t\in \R_+$).

\medskip

The projection $P_{\cal C}$ is said \emph{${\cal C}$-isotone} if $u,\;v\in \R^n,\;u\leq_{\cal C}v$ implies $P_{\cal C}u\leq_{\cal C} P_{\cal C}v.$ If $P_{\cal C}$ is ${\cal C}$-isotone, then ${\cal C}$ is called an \emph{isotone projection cone}. At this point, we would like to emphasize that in \cite{isac1986monotonicity}, the authors investigate isotone projection properties only for pointed cones. Our purpose in this section is thus to generalize their results to cones which are not necessarily pointed, hence introducing of the decomposition ${\cal C}=K\oplus L$ in Lemma \ref{ossz}, where $K$ is a pointed cone. 

\begin{theorem}\label{fo}
Let ${\cal C}\subset \R^n$ be a cone,
$${\cal C}=K\oplus L$$
with $L={\cal C}\cap (-{\cal C})$ and $K={\cal C}\cap L^\perp$.
Then, ${\cal C}$ is an isotone projection cone if and only if
the pointed cone $K\subset L^\perp$ is an isotone projection cone in $L^\perp$.
\end{theorem}

A simple geometric characterization of the isotone projection (pointed) cones was given in \cite{isac1986monotonicity}. It uses the notion of the polar of a cone. If ${\cal C}\subset \R^n$ is a cone, then the set
$${\cal C}^{\perp}=\{y\in \R^n:\;\langle x,y\rangle \leq 0, \;\forall \;x\in {\cal C}\},$$
is called the \emph{polar} of the cone  ${\cal C}$. The set ${\cal C}^{\perp}$ is obviously a cone. If the cone ${\cal C}$ is \emph{generating} in the sense that ${\cal C}-{\cal C}=\R^n$, then the polar ${\cal C}^{\perp}$ is a pointed cone. We have the following easily verifiable result:

\begin{lemma}\label{mineknev}
Suppose that ${\cal C}$ is a generating cone. Using the notations
introduced in the Theorem \ref{fo}, and denoting the polar of the pointed cone $K$
in the subspace $L^{\perp}$ by $K^{\perp}$, we have the relation
$${\cal C}^\perp=iK^{\perp},$$
where $i$ is the inclusion mapping of $L^{\perp}$ into $\R^n.$
\end{lemma}

Putting together the main result in \cite{isac1986monotonicity}, Theorem \ref{fo} and Lemma \ref{mineknev}, we have the following conclusion:

\begin{corollary}
The generating cone ${\cal C}$ is an isotone projection cone if and only if
its polar ${\cal C}^{\perp}$ is a cone generated by linearly independent
vectors forming mutually non-acute angles. 
\end{corollary}

Let us focus our attention on the monotone cone 
\begin{equation*}
{\cal {\cal C}}^-=\{b\in \R^n:\;b_1\geq b_2\geq ...\geq b_n\}
\end{equation*}
It is readily seen that ${\cal {\cal C}}^-$ is a generating cone, but not a pointed cone. Specifically, let
 \begin{equation*}
L={\cal {\cal C}}^-\cap(-{\cal {\cal C}}^-)=\{x\in \R^n:\;x_1=x_2=...=x_n\},
\end{equation*}


then $L\subset {\cal {\cal C}}^-$, the maximal subspace contained in ${\cal {\cal C}}^-$, is of dimension one. We have also that
\begin{equation}\label{subcone}
K=L^\perp \cap {\cal C}^-
\end{equation}
is an $(n-1)$-dimensional pointed cone in the hyperplane $L^\perp$ and
\begin{equation*}
{\cal {\cal C}}^-={\cal C}^-\cap (L^\perp \oplus L)=K\oplus L.
\end{equation*}
We will prove that the pointed cone $K$ given by (\ref{subcone}) is an isotone projection cone in $L^\perp$, hence the following result.

\begin{corollary}\label{coriso}
The monotone cones ${\cal C}^-$ and ${\cal C}^+$ admit isotone projections.
\end{corollary}

This result will be the basic ingredient to show the equality between the algorithms \textsf{I.I.R.} and \textsf{I.I.B.R.} presented in the next section (see Theorem \ref{theo:egalite-algos}).

\section{Iterative Isotone Regression}\label{sec:iir}
In this section, we first present the algorithm called Iterative Isotonic Regression (\textsf{I.I.R.} in short) which proceeds in alternating isotonic and antitonic regressions in a backfitting fashion. The connection with Von Neumann's algorithm and recent results in this topic will allow us to exhibit the limit of this estimator when the number of iterations goes to infinity.

\medskip

Let us first briefly recall the statistical framework and the idea behind our algorithm. We consider the nonparametric regression model 
\begin{equation}\label{eq:modele-generaliir}
Y=r(X)+\varepsilon
\end{equation}
where the random variable $X$ is, for example, uniform in $[0,1]$. Classical hypothesis for studying the consistency of an estimator of $r$ are $\E\left[Y^2\right]<\infty$ and $\E\left[\varepsilon|X\right]=0$, but this will not play a prominent role in the following since we will not investigate the statistical properties of \textsf{I.I.R}. More important, we will assume that the regression function $r$ is right-continuous and of bounded variation. Then, the Jordan decomposition asserts that such a function can be written as the sum of a non-decreasing and a non-increasing function: $r(x)=u(x)+b(x)$. Specifically, if we impose that $u$ and $b$ have singular associated Stieltjes measures and, for example, that 
\begin{equation}\label{eq:hyp-identifiabilite}
\int_0^1u(x)\,dx=\int_0^1r(x)\,dx,
\end{equation}
then the decomposition is unique. In this case, we adopt the following notation
\begin{equation}\label{eq:jordanmin}
\forall x\in[0,1]\hspace{1cm}r(x)=u^{\star}(x)+b^{\star}(x),
\end{equation}
 and we call this latter writing the Jordan minimum variation decomposition of $r$.

\subsection{Framework, notations and convergence}
In a statistical setting, we shall use a sample of i.i.d. random couples $(X_1,Y_1),\ldots,(X_n,Y_n)$ following the model (\ref{eq:modele-generaliir}) and try to estimate the regression function $r$. Viewing $u^{\star}$ and $b^{\star}$ as the components of an additive model involving two monotone terms, the idea of \textsf{I.I.R.} is to apply alternatively isotonic and antitonic regressions. Recall that $x_1=X_{(1)}<\ldots< x_n=X_{(n)}$ denotes  the ordered sample, that $y_1,\ldots, y_n$ are the corresponding observations and that $\iso(y)$ (resp. $\anti(y)$) is the isotonic (resp. antitonic) regression of $y=\left(y_1,\ldots,y_n\right)'$.
\medskip 

Besides, if $z=\left(z_1,\ldots,z_n\right)'$, the notation $\Delta(z)$ stands for the $(n-1)$ dimensional vector defined as 
\begin{equation*}
\Delta(z)=\left(z_2-z_1,\ldots,z_n-z_{n-1}\right)'.
\end{equation*} 
Considering two vectors $z$ and $\tilde{z}$, we may write $\Delta(z)\circ\Delta(\tilde{z})$ for their term-by-term (or Hadamard) product, which means 
\begin{equation*}
\Delta(z)\circ\Delta(\tilde{z})=\left((z_2-z_1)\times(\tilde{z}_2-\tilde{z}_1),\ldots,(z_n-z_{n-1})\times(\tilde{z}_n-\tilde{z}_{n-1})\right)'. 
\end{equation*}
If $\Delta(z)\circ\Delta(\tilde{z})=0$, we will say that $z$ and $\tilde{z}$ have singular variations. 
\medskip

We are now in a position to specify the Iterative Isotonic Regression algorithm.

\begin{algorithm}[H]                      
\caption{Iterative Isotonic Regression (\textsf{I.I.R.})}          
\label{algo:iir}                           
\begin{algorithmic}                   
\STATE (1) Initialization:
 \begin{align*}
\hat{b}^{(0)}&=\left(\hat{b}^{(0)}_1,\ldots,\hat{b}^{(0)}_n\right)'=0
 \end{align*} 
\STATE (2) Cycle: for $k\geq 1$\\

\begin{equation}\label{eq:cycle}
\begin{array}{ll}
\hat{u}^{(k)}&=\iso\left(y-\hat{b}^{(k-1)}\right)\\
\hat{b}^{(k)}&=\anti\left(y-\hat{u}^{(k)}\right)\\
\hat{y}^{(k)}&=\hat{u}^{(k)}+\hat{b}^{(k)}.
\end{array}
\end{equation}
\STATE (3) Iterate (2) until a stopping condition to be specified.   
\end{algorithmic}
\end{algorithm}

We prove in the Appendix (see Section \ref{sec:identifiability}) that at each iteration $k\geq 1$,
\begin{equation}\label{eq:identifiabilite}
\left\{\begin{array}{l}
\Delta\left(\hat{u}^{(k)}\right)\circ \Delta\left(\hat{b}^{(k)}\right)=0\\
\bar{\hat{u}}^{(k)}=\bar{y} \\
\bar{\hat{b}}^{(k)}=0,
\end{array}\right.
\end{equation}
where $\bar{y}:=(\sum_{i=1}^ny_i)/n$ stands for the empirical mean of $y$. These equations ensure the identifiability of the decomposition $\hat{y}^{(k)} = \hat{u}^{(k)}+\hat{b}^{(k)}$. One might indeed consider them as the translation of conditions (\ref{eq:hyp-identifiabilite}) in a discrete context. In other words, given $\hat{y}^{(k)}$, both vectors $\hat{u}^{(k)}$ and $\hat{b}^{(k)}$ are uniquely specified.

\medskip

Figure \ref{fig:iir-interpolation} illustrates the application of the \textsf{I.I.R.} algorithm on the example of the fonction $r$ which is drawn on the top left hand side of the figure. Still on the top left hand side of the figure, we have also plotted $n=100$ points $(x_i,y_i)$ to obtain our sample: for each point, one has  $y_i=r(x_i)+\varepsilon_i$, where $\varepsilon_i$ is a Gaussian centered random variable. Besides these sample points, the three other figures show the estimations $\hat{y}^{(k)}$ with $k=1, 10, 1000$ iterations.     

\begin{figure}[H]
\begin{center}
\input{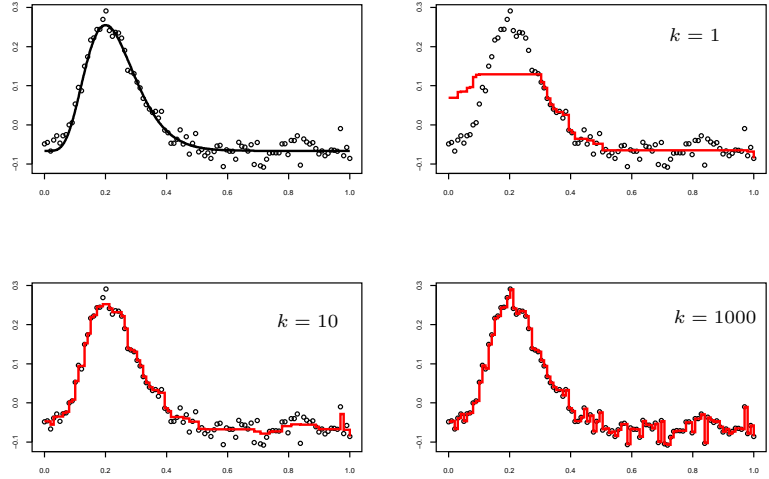}
\end{center}
\caption{Application of the \textsf{I.I.R.} algorithm for $k=1,10,1000$.}
\label{fig:iir-interpolation}
\end{figure}

One can see that for each iteration the method fits a piecewise constant function to the data and that increasing the number of iterations leads to increase the number of level sets (clusters). It also appears on this example that iterating the algorithm leads to the interpolation of the original data. This property is always true, as established by the following result.

\begin{proposition}\label{theo:iir-interpolation1}
With the previous notations, one has
$$\lim_{k\rightarrow \infty}\hat{y}^{(k)}=y.$$
\end{proposition}  

\noindent\begin{proof}
In the following, $y$ is held fixed and, as usual, $y+{\cal C}^+$ stands for the translated cone 
$$y+{\cal C}^+=\{y+u,u\in{\cal C}^+\}.$$
All the notations are recalled on figure \ref{fig:von-neumann1}. To understand the geometrical forces driving Proposition \ref{theo:iir-interpolation1}, this figure also provides a very simple interpretation of the algorithm, as it illustrates that the sequences of vectors $\hat{u}^{(k)}$ and $y-\hat{b}^{(k)}$ might be seen as alternate projections on the cones ${\cal C}^+$ and $y+{\cal C}^+$. In what follows, we justify this illuminating geometric interpretation in a rigorous way, and we explain its key role in the proof of the convergence as $k$ goes to infinity. 

\medskip

First, by definition, we have $\hat{u}^{(1)}=P_{{\cal C}^+}(y)$, and a moment's thought shows that 
$$P_{y+{\cal C}^+}(\hat{u}^{(1)})=y+P_{{\cal C}^+}(\hat{u}^{(1)}-y)=y-P_{{\cal C}^-}(y-\hat{u}^{(1)}).$$
Then, coming back to the very definition of  $\hat{b}^{(1)}=P_{{\cal C}^-}(y-\hat{u}^{(1)})$ leads to
$$y-\hat{b}^{(1)}=P_{y+{\cal C}^+}(\hat{u}^{(1)}).$$
In the same manner, since $\hat{u}^{(2)}=P_{{\cal C}^+}(y-\hat{b}^{(1)})$, we get 
$$y-\hat{b}^{(2)}=y-P_{{\cal C}^-}(y-\hat{u}^{(2)})=y+P_{{\cal C}^+}(y-\hat{u}^{(2)})=P_{y+{\cal C}^+}(\hat{u}^{(2)})$$
More generally, denoting $\hat{b}^{(0)}=0$, this yields for all $k\geq 1$
\begin{equation*}
\hat{u}^{(k)}=P_{{\cal C}^+}(y-\hat{b}^{(k-1)}) \qquad \textrm{and} \qquad \hat{b}^{(k)}=P_{y+{\cal C}^+}(\hat{u}^{(k)}).
\end{equation*}

\begin{figure}[H]
\begin{center}
\input{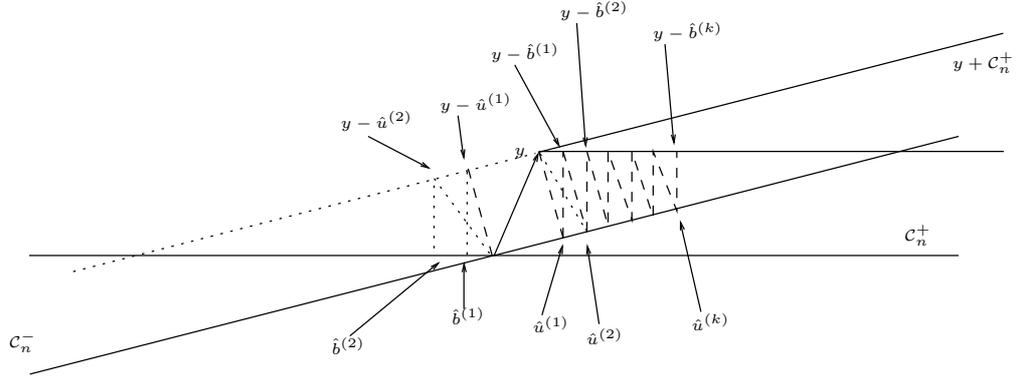}
\end{center}
\caption{Interpretation of the \textsf{I.I.R.} method as a Von Neumann's algorithm.}
\label{fig:von-neumann1}
\end{figure}

It remains to invoke Theorem 4.8 in \cite{bauschke1994dykstra} to conclude that  
$$(y-\hat{b}^{(k)})-\hat{u}^{(k)}=y-\hat{y}^{(k)}\xrightarrow[k\to\infty]{} 0,$$
which ends the proof of Proposition \ref{theo:iir-interpolation1}.
\end{proof}

A few remarks are in order:

\begin{enumerate}
\item The take-home message here is that iterating the algorithm leads to overfitting, which is of course not desired in practice. Consequently, a stopping criterion must be applied in order to avoid this phenomenon. In this aim, we have tested several pratical rules, and this will be the object of Section \ref{sec:simulation}.
\item Note that Proposition \ref{theo:iir-interpolation1} ensures the convergence of the sum $\hat{u}^{(k)}+\hat{b}^{(k)}$ but does not say anything about the convergence of its individual terms $\hat{u}^{(k)}$ and $\hat{b}^{(k)}$. However, Corollary 4.9 in \cite{bauschke1994dykstra} asserts that the sequences $\hat{u}^{(k)}$ and $\hat{b}^{(k)}$ are also convergent. The specification of these limits will be possible in our situation thanks to the introduction of another equivalent algorithm, called \textsf{I.I.B.R}. This will be the topic of the next subsection. 
\item In \cite{mammen2007additive}, Mammen and Yu consider isotonic functions in the multivariate additive model (\ref{eq:additive-models}). They rely on the analysis of Dykstra's algorithm sequences \cite{dykstra1983algorithm} to show that for fixed $n$ the backfitting procedure converges to the solution of 
$$\arg\min\sum_{j=1}^n\left(y_j-u_j\right)^2$$
where $u$ is the sum of $d$ vectors in ${\cal C}^+$. Hence, our case is rather comparable to the one considered by these authors when $d=2$. Correspondingly, in our setting, one can see the vectors $y-\hat{b}^{(k-1)}-\hat{u}^{(k)}$ and $y-\hat{u}^{(k)}-\hat{b}^{(k)}$ as Dykstra's sequences and consider Mammen and Yu's alternate projections on the polar cones of ${\cal C}^+$ and ${\cal C}^-$ to prove Proposition \ref{theo:iir-interpolation1}. However, as we are alternating projections onto opposite cones, considering  $\hat{r}^{(k)}$ and $y-\hat{b}^{(k)}$ yields the easier Von Neumann's type interpretation given in our proof.
\end{enumerate}    

\subsection{The connection with Iterative Isotonic Bias Reduction}
In this section we propose another algorithm inspired by bias reduction techniques in regression \cite{buhlmann2003boosting}\cite{cornillon2011}. In a nutshell, the idea here is to work on the updated residuals $y-\hat{y}^{(k)}$ in order to refine the estimator at each step. Specifically, this algorithm takes the following form. 
\begin{algorithm}[H]                      
\caption{Iterative Isotonic Bias Reduction  (\textsf{I.I.B.R.})}          
\label{algo:bir}                           
\begin{algorithmic}                   %
\STATE (1) Initialization:
\begin{align*}
\hat{y}^{(0)}&=0
\end{align*} 
\STATE (2) Cycle: for $k\geq 1$
\begin{equation*}
\begin{array}{ll}
\tilde{u}^{(k)}&=\iso\left(y-\hat{y}^{(k-1)}\right)\\
\tilde{b}^{(k)}&=\anti\left(y-\hat{y}^{(k-1)}-\hat{u}^{(k)}\right)\\
\end{array}
\end{equation*}
Updating:
\begin{equation*}
\hat{y}^{(k)}=\hat{y}^{(k-1)}+\tilde{u}^{(k)}+\tilde{b}^{(k)}
\end{equation*}
\STATE (3) Iterate (2) until a stopping condition to be specified.   
\end{algorithmic}
\end{algorithm}

Interestingly, it turns out that algorithms \textsf{I.I.R.} and \textsf{I.I.B.R.} coincide. This remarkable fact, which is the purpose of the next theorem, deeply relies on the property of isotonicity of the projection onto the cones we consider. As a by-product, it will allow us to make precise the individual limits of the sequences $\hat{u}^{(k)}$ and $\hat{b}^{(k)}$.

\begin{theorem}\label{theo:egalite-algos}
With the previous notations, for all $k\geq 1$,
$$\hat{u}^{(k)}=\sum_{j=1}^k \tilde{u}^{(j)} \qquad \textrm{and} \qquad \hat{b}^{(k)}=\sum_{j=1}^k \tilde{b}^{(j)}.$$
\end{theorem}

Proposition \ref{theo:iir-interpolation1} ensures that $\hat{y}^{(k)}$ tends to $y$ when $k$ goes to infinity. Thanks to Theorem \ref{theo:egalite-algos} we will go one step further and show in Corollary \ref{theo:iir-interpolation2} that the individual terms $\hat{u}^{(k)}$ and $\hat{b}^{(k)}$ tend to some explicit limits $u_{y}^{\star}$ and $b_{y}^{\star}$. These limits are simply the discrete analogous of the Jordan minimum variance decomposition of a function with bounded variation (\ref{eq:jordanmin}) for the vector $y$. As will be stated below, they are indeed characterized by $y=u_y^{\star}+b_y^{\star}$ and 
$$\Delta(u_y^{\star})\circ\Delta(b_y^{\star})=0,$$
the empirical means of $y$ and $u_y^{\star}$ being the same.

\medskip

Before proceeding, let us illustrate this on the example of figure \ref{fig:convergence-individuel}. The left-hand side represents the vector $y$ and the decomposition $y=u_y^{\star}+b_y^{\star}$. One can see that $$y_{i+1}-y_i>0 \Rightarrow u_{y,i+1}^{\star}-u_{y,i}^{\star}=y_{i+1}-y_i \textrm{ and } b_{y,i+1}^{\star}-b_{y,i}^{\star}=0$$
and conversely, 
$$y_{i+1}-y_i<0 \Rightarrow u_{y,i+1}^{\star}-u_{y,i}^{\star}=0 \textrm{ and }b_{y,i+1}^{\star}-b_{y,i}^{\star}=y_{i+1}-y_i,$$
which, in the discrete case, amounts to say that $u^{\star}$ and $b^{\star}$ have singular associated Stieltjes measures in equation  (\ref{eq:jordanmin}).
The right-hand side of the figure illustrates that when the number of iterations $k$ goes to infinity, $\hat{u}^{(k)}$ and $\hat{b}^{(k)}$ converge respectively to $u_{y}^{\star}$ and $b_{y}^{\star}$.

\begin{figure}[H]
\begin{center}
\input{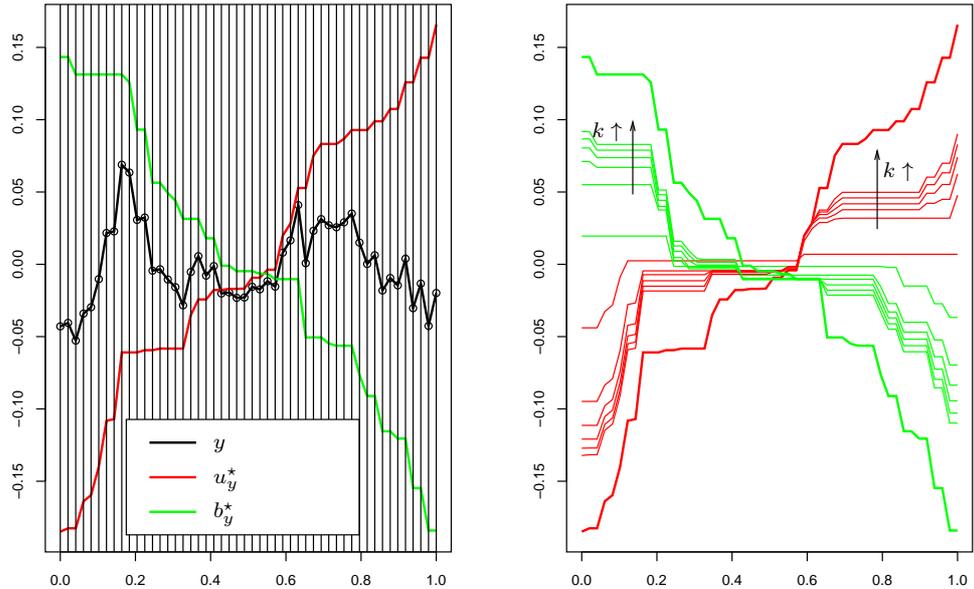}
\end{center}
\caption{Respective convergences of $\hat{u}^{(k)}$ and $\hat{b}^{(k)}$ to $u_{y}^{\star}$ and $b_{y}^{\star}$.}
\label{fig:convergence-individuel}
\end{figure}

\begin{corollary}\label{theo:iir-interpolation2}
The sequences $\left(\hat{u}^{(k)}\right)_{k\geq 1}$ and $\left(\hat{b}^{(k)}\right)_{k\geq 1}$ admit the following limits
$$\lim_{k\rightarrow \infty}\hat{u}^{(k)}=u_y^{\star} \qquad \textrm{ and }\qquad \lim_{k\rightarrow \infty}\hat{b}^{(k)}=b_y^{\star}$$
where $u_y^{\star}$ and $b_y^{\star}$ are such that $y=u_y^{\star}+b_y^{\star}$, $\bar{u_y^{\star}}=\bar{y}$, $\bar{b_y^{\star}}=0$ and $\Delta(y)=\Delta(u_y^{\star})\circ\Delta(b_y^{\star})$. 
\end{corollary}

\begin{proof}
In order to lighten the notations a bit, let us assume that the empirical mean of $y$ is equal to zero. Consequently, all the intermediates of both algorithms have also zero mean, as well as $u_{y}^{\star}$ and $b_{y}^{\star}$. This allows us to work in the hyperplane ${\cal E}$, that means the subspace of $\R^n$ consisting in all zero mean vectors. In this hyperplane, we introduce the norm $\nv$ to quantify the variations of a vector:
$$
 \begin{array}{llll}
\nv: & {\cal E} & \rightarrow & \R^+\\
   & y & \mapsto & \nv(y)=\sum_{i=1}^{n-1} |y_{i+1}-y_i|
\end{array}
$$

Then, it is routine to check that for any vector $z\in{\cal E}$ and any decomposition $z=u_z+b_z$ with $u_z$ and $b_z$ both in ${\cal E}$, we have the following equivalence 
\begin{equation}\label{eq:equivalence}
\Delta(u_z)\circ\Delta(b_z)=0\ \Leftrightarrow\ \nv(z)=\nv(u_z)+\nv(b_z).
\end{equation}

Now, Corollary \ref{cor:identifiabilite} in Section \ref{sec:identifiability} ensures that for all $k\geq 1$, $\hat{u}^{(k)}$ and $\hat{b}^{(k)}$ have singular variations (one can also see that on figure \ref{fig:convergence-individuel}) which implies that for all $k\geq 1$ 
\begin{equation*}
 \nv(\hat{y}^{(k)})=\nv(\hat{u}^{(k)})+\nv(\hat{b}^{(k)}).
\end{equation*}

Accordingly, we just have to justify that the limits $u^{\infty}:=\lim_{k\rightarrow \infty}\hat{u}^{(k)}$ and $b^{\infty}:=\lim_{k\rightarrow \infty}\hat{b}^{(k)}$ satisfy $\nv(u^{\infty})+\nv(b^{\infty})=\nv(y)$ to deduce that $\Delta(u^{\infty})\circ\Delta(b^{\infty})=0$, hence $u^{\infty}=u_y^{\star}$ and $b^{\infty}=b_y^{\star}$.
The existence of the limits implies that for all $i\in\{1,\dots,n\}$, 
$$ u^{\infty}_i:=\lim_{k\rightarrow \infty}\hat{u}^{(k)}_i \qquad \textrm{and}\qquad  b^{\infty}_i:=\lim_{k\rightarrow \infty}\hat{b}^{(k)}_i$$
are well-defined. Then, the continuity of the norm $\nv$ leads to
\begin{equation*}
\lim_{k\rightarrow \infty}\nv(\hat{u}^{(k)})=\nv(u^{\infty}) \qquad \textrm{and} \qquad \lim_{k\rightarrow \infty}\nv(\hat{b}^{(k)})=\nv(b^{\infty}),
\end{equation*}
and the same relation holds for $\hat{y}^{(k)}$ and $y$. Thus, 
$$
\begin{array}{lll}
\nv(y)&=\nv \left(\lim_{k\rightarrow \infty}\{\hat{u}^{(k)}+\hat{b}^{(k)}\}\right)&\\
    &=\lim_{k\rightarrow \infty}\nv \left(\hat{u}^{(k)}+\hat{b}^{(k)}\right)&\\
    &=\lim_{k\rightarrow \infty}\left\{\nv(\hat{u}^{(k)})+\nv(\hat{b}^{(k)})\right\}&\\
    &=\lim_{k\rightarrow \infty}\nv (\hat{u}^{(k)})+\lim_{k\rightarrow \infty}\nv(\hat{b}^{(k)})&\\
    &=\nv\left(\lim_{k\rightarrow \infty} \hat{u}^{(k)}\right)+\nv\left(\lim_{k\rightarrow \infty} \hat{b}^{(k)}\right)&\\
    &=\nv(u^{\infty})+\nv(b^{\infty}).&
\end{array}
$$
From equivalence (\ref{eq:equivalence}), we deduce that $u^{\infty}=u_y^{\star}$ and $b_y^{\star}=b^{\infty}$, where $y=u_y^{\star}+b_y^{\star}$ is the Jordan minimum variance decomposition of $y$.  
 
\end{proof}

\section{Discussion}\label{sec:simulation}
As increasing the number of iterations  leads to overfit the data, iterating the procedure until convergence is not desirable. This brings up to the choice of a stopping rule which could be used in practice. Viewing the latter question as a model selection issue suggests stopping criteria based for example on Akaike Information Criterion (AIC, see \cite{akaike1973}), modified AIC criterion (AICc, see \cite{hurvich1998smoothing}), Bayesian Information Criterion (BIC, see \cite{schwarz1978estimating}) and Generalized Cross Validation (GCV, see \cite{craven1978smoothing}). 

\medskip

For a linear smoother $\hat{r}_{\lambda}=S_{\lambda}(y)$, with $\lambda$ the smoothing parameter and $S_{\lambda}$ the smoothing matrix, these selectors can be written in the common form   
\begin{equation}\label{eq:stop-commonform}
\arg\min_{\lambda}\left\{\log \frac{1}{n}\RSS_{\lambda}+\phi(p_{\lambda})\right\}
\end{equation}
where $\RSS$ denotes the residual sum of squares and $\phi(.)$ is an increasing function of the number of degrees of freedom  $p_{\lambda}$ of the smoother. One usually takes $p_{\lambda}=\trace(S_{\lambda})$ or $p_{\lambda}=\trace(S_{\lambda}S_{\lambda}^{t})$ (see \cite{buja1989linear}, section 2.7.3) and according to the various criteria mentioned above 
\begin{equation}\label{eq:phi-criteres}  
\begin{array}{ll}
\phi_{AIC}(p_{\lambda})&=2\frac{p_{\lambda}}{n}\\
\phi_{BIC}(p_{\lambda})&=\frac{p_{\lambda}}{n}\log n\\ 
\phi_{AICc}(p_{\lambda})&=1+2\frac{p_{\lambda}+1}{n-p_{\lambda}-2}\\
\phi_{GCV}(p_{\lambda})&=-2\log\left(1-\frac{p_{\lambda}}{n}\right).
\end{array}
\end{equation}

Thus, equation (\ref{eq:stop-commonform}) leads to a tradeoff between the goodness of fit and a penalization of high model complexity.   Although isotonic regression is not a linear smoother, we refer to Meyer and  Woodroofe \cite{meyer2000} to consider that the number of distinct values among the fitted vector provides the effective dimension of the model. Taking this into account and considering that increasing the number of iterations tends to reduce the residual sum of squares but raises the complexity of the model, a natural extension for iterative isotonic regression is to replace $p_{\lambda}$ by the number of sets of the fitted vector $\hat{r}^{(k)}$ in (\ref{eq:stop-commonform}) and solve
\begin{equation}\label{eq:criteria-iir}
\arg\min_{k\in {\cal K}}\left\{\log \frac{1}{n}\RSS_{k}+\phi(p_{k})\right\}
\end{equation}
over a grid ${\cal K}$ of iterations.

\medskip

We have applied and compared these stopping rules for iterative isotonic regression through simulated data. It appears that AICc shows the best performances among the three other criteria for most investigated cases. Then, for this specific stopping criterion, we have compared iterative isotonic regression with non parametric competitors, namely local polynomial regression (\textsf{R} package \textsf{locpol}), and smoothing spline regression (\textsf{R} package \textsf{stats}, function \textsf{smooth.spline}). For further details on this topic, we refer the interested reader to the \textsf{R} package dedicated to \textsf{I.I.R.} at the following address:
{\small
\begin{verbatim}
www.sites.univ-rennes2.fr/laboratoire-statistique/JEGOU/index.html
\end{verbatim}}

\medskip

The take-home message is that \textsf{I.I.R.} can not compete with spline or local polynomial regression for smooth functions. However, when the functions have discontinuities, our estimator compares favorably with his competitors, in particular when the sample size increases. This suggests that our method could be used to locate  discontinuities in a regression framework. Applications arise for example in genomic where the detection of breakpoints from Array Comparative Genomic Hybridization (array CGH) profiles is of crucial importance to identify genes involved in cancer progression \cite{hupe2004}.

\section{Appendix}\label{sec:proofs}
\subsection{Proof of Lemma \ref{ossz}}

The relation (\ref{ortossz}) follows directly  from
$${\cal C}={\cal C}\cap (L^\perp \oplus L).$$
It is known (see \cite{Zarantonello1971}) that the projection $P_{\cal C}x$ of $x$ onto the cone ${\cal C}$ is characterized by the couple of relations:
\begin{equation}\label{elso}
\lang x-P_{\cal C}x,y\rang \leq 0,\;\forall \; y \in {\cal C},
\end{equation}
and
\begin{equation}\label{masod}
\lang x-P_{\cal C}x,P_{\cal C}x\rang = 0.
\end{equation}
Hence, we have to verify the above relations for $P_Kx_k+x_l$ instead of $P_{\cal C}x$. By the relation (\ref{ortossz}), $P_Kx_k+x_l \in {\cal C}.$ Then, take an arbitrary $y\in {\cal C}$ represented by (\ref{ortossz}) in the form
$$y=y_k+y_l$$
with $y_k\in K$ and $y_l\in L$. Then, we have
$$ \lang x_k+x_l-(P_Kx_k+x_l),y_k+y_l\rang = \lang x_k-P_Kx_k,y_k\rang \leq 0,\;\;\forall \;y=y_k+y_l\in {\cal C},$$
because $y_l$ is perpendicular to $x_k-P_Kx_k\in L^\perp$, and because of the relation similar to (\ref{elso}) characterizing the projection of $x_k$ onto the pointed cone $K$ in $L^\perp$. Thus, relation (\ref{elso}) holds for $P_Kx_k+x_l$ in place of $P_{\cal C}x.$ We further have 
$$\lang x_k+x_l-(P_Kx_k+x_l),P_Kx_k+x_l\rang =\lang x_k-P_Kx_k,P_Kx_k\rang =0$$
because $x_l$ is perpendicular to $x_k-P_Kx_k$ and because of the relation similar to (\ref{masod}) applied to $x_k\in L^\perp$ and its projection onto $K$. The obtained relation is exactly (\ref{masod}) for $P_Kx_k+x_l$ instead of $P_{\cal C}x.$

\subsection{Proof of Theorem \ref{fo}}

Take $u,\;v\in L^\perp$. Then, $u\leq_K v$ is equivalent to $u\leq_{\cal C} v$. If $P_{\cal C}$ is ${\cal C}$-isotone, then $u\leq_{\cal C} v$ implies by Lemma \ref{ossz}
\begin{equation*}
P_Kv-P_Ku=P_{\cal C}v-P_{\cal C}u\in {\cal C}.
\end{equation*}
Since $P_Ku,\:P_Kv\in L^\perp$, it follows that
$$P_Kv-P_Ku\in L^\perp \cap {\cal C}=K.$$
The obtained relation shows that $P_K$ is $K$-isotone, concluding the proof of the necessity of the theorem.

\medskip

Suppose now that $P_K$ is $K$-isotone and take $u,\;v\in \R^n$ with $u\leq_{\cal C} v$. If $u=u_k+u_l$ and $v=v_k+v_l$ with $u_k,\;v_k\in L^\perp$, and $u_l,\;v_l\in L$, then using formula (\ref{ortossz})
$$v-u=v_k-u_k+v_l-u_l\in K\oplus L$$
and hence $v_k-u_k\in K$, that is $u_k\leq_K v_k$ and by the $K$-isotonicity of $P_K$ it follows that
$$P_Kv_k-P_Ku_k\in K.$$
Hence, using formula (\ref{projek}) we have
$$P_{\cal C}v-P_{\cal C}u=P_Kv_k+v_l-P_Ku_k-u_l= P_Kv_k-P_Ku_k +v_l-u_l\in K\oplus L={\cal C}.$$
That is $P_{\cal C}u\leq_{\cal C} P_{\cal C}v$, which concludes the isotonicity of $P_{\cal C}$.

\subsection{Proof of Corollary \ref{coriso}}\label{sec:proof-cor2}

It clearly suffices to prove that the monotone cone ${\cal C}^-=\{b\in \R^n:\;b_1\geq b_2\geq ...\geq b_n\}$ admits an isotone projection. For this, we have to introduce some notations. Let us first take the following base in $\R^n$: 
$$e_1=(1,0,...,0), e_2=(1,1,0,...,0),\dots,e_{n-1}=(1,...,1,0),e_n=(1,1,...,1).$$
Then, an arbitrary element $b=(b_1,...,b_n)\in \R^n$ can be represented in the form
\begin{equation}\label{kifejez}
b=(b_1-b_2)e_1+(b_2-b_3)e_2+...+(b_{n-1}-b_n)e_{n-1}+b_ne_n,
\end{equation}
the relation $b\in {\cal C}^-$ being equivalent with 
\begin{equation}\label{ekelem}
b_{j-1}-b_j\geq 0,\;\;j=2,...,n.
\end{equation}
Let us consider further the following base in $L^\perp$:
$$e'_1=(n-1,-1,-1,...,-1),e'_2=(n-2,n-2,-2,...,-2),\dots,e'_{n-1}=(1,...,1,-(n-1)).$$
The following notation is standard in the convex geometry and ordered vector space theory: If $M\subset \R^n$ is a non-empty set, then let
$$\cone(M)=\{t_1n_1+...+t_kn_k:\;n_i\in M,\;t_i\in \R_+=[0,+\infty),\;i=1,...,k;\;k\in \N\}.$$
The set $\cone(M)$ is the minimal cone containing the set $M$ and it is called \emph{the cone generated by $M$}. Denoting 
\begin{equation*}
L={\cal C}^-\cap(-{\cal C}^-)=\{b\in \R^n:\;b_1=b_2=...=b_n\},
\end{equation*}
and $K=L^\perp \cap {\cal C}^-$, we will see next that
\begin{equation}\label{ak}
K=\cone(\{e'_1,...,e'_{n-1}\}).
\end{equation}
Since $e'_j\in {\cal C}^-\cap L^\perp$, we have obviously that
\begin{equation}\label{benne}
\cone(\{e'_1,...,e'_{n-1}\})\subset K.
\end{equation}
Comparing the vectors $e_i$ and $e'_j$ we get
\begin{equation}\label{comparing}
\frac{1}{n-j+1}(e'_j+e_n)=e_j,\;\;j=1,...,n-1. 
\end{equation}
By substitution of $e_j,\;j=1,...,n-1$, the representation (\ref{kifejez}) of $b$ becomes
\begin{equation}\label{kifejez1}
b =(b_1-b_2)\frac{1}{n}(e'_1-e_n)+(b_2-b_3)\frac{1}{n-1}(e'_2+e_n)+...+(b_{n-1}-b_n)\frac{1}{2}(e'_{n-1}+e_n)+b_ne_n
\end{equation}
Suppose now that $b\in {\cal C}^-$, that is, relations (\ref{ekelem}) hold. Then,  the coefficients of $e'_j,\;j=1,...,n-1$ in its representation (\ref{kifejez1}) are non-negative. Thus, we have
\begin{equation}\label{kesz}
	b\in {\cal C}^- \Leftrightarrow b =\sum_{j=1}^{n-1} t_je'_j+t_ne_n,\textrm{ }t_j\in \R_+,
	\textrm{ }j=1,...,n-1,\textrm{ }t_n\in \R.
\end{equation}
In particular, if $b\in K$, then, by \eqref{subcone}, we have $b\in L^\perp$. Hence, by multiplying (\ref{kesz}) scalarly by $e^n$ and by using $\lang b,e_n\rang=0$ (which follows from $b\in L^\perp$ and $e_n\in L$) and $\lang e'_j,e_n\rang=0$ (which follows from $e'_j\in L^\perp$ and $e_n\in L$), we get $t_n=0$. This reasoning shows that
$$K \subset \cone (\{e'_1,...,e'_{n-1}\}),$$
inclusion which together with (\ref{benne}) proves (\ref{ak}).
\medskip
We consider now the vectors
$$u_1=(-1,1,0,...,0),u_2=(0,-1,1,0,...,0),\dots,u_{n-1}=(0,...,0,-1,1).$$
Then, $u_i\in L^\perp,\;i=1,...,n-1,$ and we have
\begin{equation}\label{polar}
\lang u_i,e'_j\rang=0 \;\textrm{if} \; i\not= j,\;\;\lang u_i,e'_i\rang <0,\;i,j=1,...,n-1.
\end{equation}
According to the reasonings in \cite{NemethNemeth2009} the relations (\ref{polar}) show that
$$\cone (\{u_1,...,u_{n-1}\})$$
is the polar of $K$ in the subspace $L^\perp$. Further, we have
$$\lang u_i,u_j\rang \leq 0 \;\;\textrm{if} \; i\not=j.$$
By the main result in \cite{isac1986monotonicity} this shows that $K$ is an isotone projection pointed cone in $L^\perp.$

\subsection{Proof of Theorem \ref{theo:egalite-algos}}\label{sec:egalite-algos}

First we prove that for all $k\geq 1$, 
\begin{equation}\label{eq:intermediaire}
 \hat{u}^{(k+1)}-\hat{u}^{(k)} \in{\cal C}^+ \qquad \textrm{and} \qquad\hat{b}^{(k+1)}-\hat{b}^{(k)} \in {\cal C}^-.
\end{equation}
For $k=1$, we have 
$$\hat{u}^{(2)}-\hat{u}^{(1)}=\iso(y-\hat{b}^{(1)})-\iso(y).$$
Since $-\hat{b}^{(1)}\in{\cal C}^+$ and as ${\cal C}^+$ is an isotone projection cone by Corollary \ref{coriso}, we deduce that $\hat{u}^{(2)}-\hat{u}^{(1)}$ belongs to ${\cal C}^+$. In the same manner, 
$$\hat{b}^{(2)}-\hat{b}^{(1)}=\anti(y-\hat{u}^{(2)})-\anti(y-\hat{u}^{(1)}),$$
which is equivalent to 
$$\hat{b}^{(2)}-\hat{b}^{(1)}=\anti\left(y-\hat{u}^{(1)}-(\hat{u}^{(2)}-\hat{u}^{(1)}) \right)-\anti(y-\hat{u}^{(1)}).$$
Now, as was just noticed, $-(\hat{u}^{(2)}-\hat{u}^{(1)})$ belongs to ${\cal C}^-$, which is also an isotone projection cone, so that $\hat{b}^{(2)}-\hat{b}^{(1)}$ is indeed in ${\cal C}^-$. We can iterate this reasoning. For example, at the next step: 
$$\hat{u}^{(3)}-\hat{u}^{(2)}=\iso(y-\hat{b}^{(2)})-\iso(y-\hat{b}^{(1)})=\iso\left(y-\hat{b}^{(1)}-(\hat{b}^{(2)}-\hat{b}^{(1)})\right)-\iso(y-\hat{b}^{(1)}),$$
and we may go on with the same arguments as before.

\medskip 

Next, we prove the desired result by induction on $k$. At the first step ($k=1$), both algorithms clearly coincide. Now let us assume that it is still true at step $k\geq 1$, which means
$$ \hat{y}^{(k)}=\hat{u}^{(k)}+\hat{b}^{(k)} \qquad \textrm{with}\ \hat{u}^{(k)}=\sum_{j=1}^{k}\tilde{u}^{(j)} \textrm{ and } \hat{b}^{(k)}=\sum_{j=1}^{k}\tilde{b}^{(j)}.$$
Our objective is to prove that $\hat{u}^{(k+1)}=\sum_{j=1}^{k+1}\tilde{u}^{(j)}$. For this, let us show that
\begin{equation}\label{eq:utile1}
\|y-\hat{b}^{(k)}-\hat{u}^{(k+1)}\|=\|y-\hat{y}^{(k)}-\tilde{u}^{(k+1)}\|.
\end{equation}
Due to the fact that $\hat{u}^{(k+1)}$ is the best non-decreasing approximation of $y-\hat{b}^{(k)}$, and since $\hat{u}^{(k)}+\tilde{u}^{(k+1)}$ itself is a non-decreasing sequence, we deduce that
$$ \|y-\hat{b}^{(k)}-\hat{u}^{(k+1)}\|\leq \|y-\hat{b}^{(k)}-(\hat{u}^{(k)}+\tilde{u}^{(k+1)})\|=\|y-\hat{y}^{(k)}-\tilde{u}^{(k+1)}\|.$$
Now, as $\tilde{u}^{(k+1)}$ is the best non-decreasing approximation of $y-\hat{y}^{(k)}$, it is in particular a better approximation than $\hat{u}^{(k+1)}-\hat{u}^{(k)}$, this latter being itself non-decreasing as was just established above (see (\ref{eq:intermediaire})). Thus, we are led to
$$\|y-\hat{y}^{(k)}-\tilde{u}^{(k+1)}\|\leq \|y-\hat{y}^{(k)}-(\hat{u}^{(k+1)}-\hat{u}^{(k)})\|= \|y-\hat{b}^{(k)}-\hat{u}^{(k+1)}\|.$$
Putting all the pieces together, (\ref{eq:utile1}) is proved, and we get
$$\|y-\hat{y}^{(k)}-\tilde{u}^{(k+1)}\|=\|y-\hat{b}^{(k)}-\hat{u}^{(k+1)}\|= \|y-\hat{y}^{(k)}-(\hat{u}^{(k+1)}-\hat{u}^{(k)})\|.$$
This indicates that $\tilde{u}^{(k+1)}$ et $\hat{u}^{(k+1)}-\hat{u}^{(k)}$ both realize the minimal distance to $y-\hat{y}^{(k)}$. As a consequence,
$$\tilde{u}^{(k+1)}=\hat{u}^{(k+1)}-\hat{u}^{(k)} \Leftrightarrow \hat{u}^{(k+1)}=\hat{u}^{(k)}+\tilde{u}^{(k+1)},$$
and finally
$$\hat{u}^{(k+1)}=\sum_{j=1}^{k+1}\tilde{u}^{(j)}.$$
The same arguments may be repeated to establish that $\hat{b}^{(k+1)}=\sum_{j=1}^{k+1}\tilde{b}^{(j)}$. Details are omitted.

\subsection{Proof of identifiability conditions}\label{sec:identifiability}

The purpose of this section is to prove the identifiability conditions (\ref{eq:identifiabilite}) for the \textsf{I.I.R.} algorithm, that is, for all $k\geq 1$ (see also figure \ref{fig:illustration-identifiabilite})
\begin{equation*}
\left\{\begin{array}{l}
\Delta\left(\hat{u}^{(k)}\right)\circ \Delta\left(\hat{b}^{(k)}\right)=0\\
\bar{\hat{u}}^{(k)}=\bar{y} \\
\bar{\hat{b}}^{(k)}=0,
\end{array}\right.
\end{equation*}

\begin{figure}[H]
\begin{center}
\input{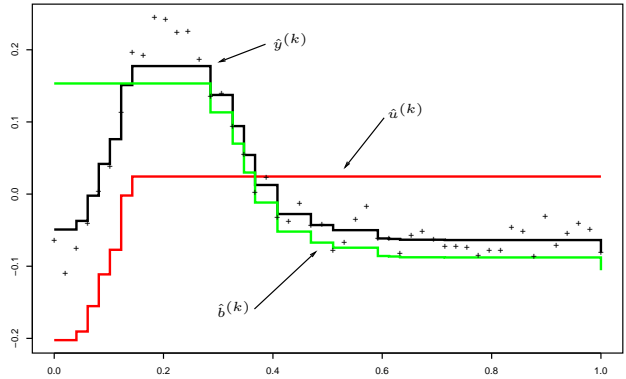}
\end{center}
\caption{Illustration of identifiability conditions.}
\label{fig:illustration-identifiabilite}
\end{figure}
  
The two last equations rely on a characterization of metric projections which was already mentioned in the proof of Lemma \ref{ossz}. From equations (\ref{elso}) and  (\ref{masod}), we know that for any $y\in \R^n$, the vector $u \in {\cal C}^+$ (resp. ${\cal C}^-$) is the isotonic (resp. antitonic) regression of $y$ if and only if 
\begin{equation*}
\langle y-u,u\rangle= 0
\end{equation*}
and 
\begin{equation}\label{eq:caracterisation1bis}
\forall v\in {\cal C}^+ (\textrm{resp. }{\cal C}^-),\ \langle y-u,v\rangle\leq 0.
\end{equation}

Taking successively $v=(1/n,\ldots,1/n)'\in {\cal C}^+\cap {\cal C}^-$ and $v=(-1/n,\ldots,-1/n)'\in {\cal C}^+\cap {\cal C}^-$ in (\ref{eq:caracterisation1bis}) leads to
\begin{equation*}
\frac{1}{n}\sum_{i=1}^n\iso(y)_i=\frac{1}{n}\sum_{i=1}^n\anti(y)_i=\overline{y},
\end{equation*}
which proves that for all $k\geq 1$, $\bar{\hat{u}}^{(k)}=\bar{y}$ and $\bar{\hat{b}}^{(k)}=0$. To prove that $\Delta\left(\hat{u}^{(k)}\right)\circ \Delta\left(\hat{b}^{(k)}\right)=0$, we will need the following result.

\begin{proposition}\label{pro:intermediaire-identifiabilite}
Let $y\in \R^n$, then  
\begin{equation}\label{eq:pro1}
u=\iso(y)\ \textrm{ and }\ b=\anti(y-u)\ \Rightarrow\ \Delta(u)\circ\Delta(b)=0
\end{equation}
and 
\begin{equation}\label{eq:pro2}
 b=\anti(y)\  \textrm{ and }\  u=\iso(y-b)\ \Rightarrow\ \Delta(u)\circ\Delta(b)=0.
\end{equation}
\end{proposition}

\noindent \begin{proof}
Let us prove for example that
$$u_i<u_{i+1} \Rightarrow b_i=b_{i+1}.$$
We first establish that after an isotone regression, the last observation $y_i$ of a cluster $A_i$ is always lower than or equal to $u_i$. For this, we use a proof by contradiction, arguing that a situation like the one depicted in figure \ref{fig:proposition-identifiabilite} is impossible.  

\begin{figure}[H]
\begin{center}
\input{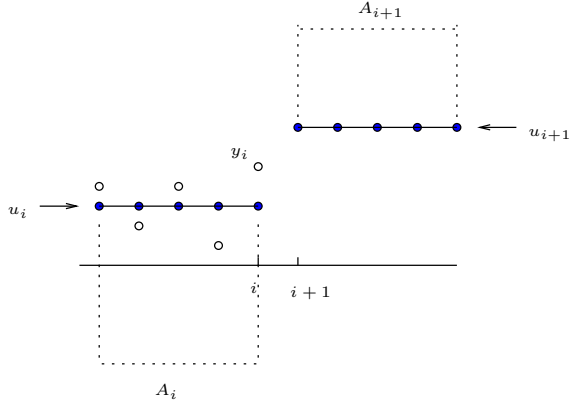}
\end{center}
\caption{Notations for the clusters $A_i$ and $A_{i+1}$.}
\label{fig:proposition-identifiabilite}
\end{figure}

Let us denote $A_i$ the cluster with last element $y_i$. For the sake of simplicity, $u_i$ stands for the last common value to all elements of cluster $A_i$ after the isotone regression. From the properties of the Pool Adjacent Violators Algorithm (see for example \cite{barlow1972statistical}), it is well-known that $u_i$ satisfies the following equation 
$$ u_i=\frac{1}{|A_i|}\sum_{y_j \in A_i}y_j,$$
where $|A_i|$ is the cardinal of $A_i$. Equivalently, the next cluster is denoted $A_{i+1}$ and the corresponding isotonic value $u_{i+1}$. 

\medskip

Now, let us assume that $u_i<y_i$, and denote 
$$c_{i-1}=\frac{1}{|A_i|-1}\sum_{y_j \in A_i,y_j\neq y_i}y_j$$ 
so that, clearly, $c_{i-1}<u_i$. Besides, it is readily seen that $y_i< u_{i+1}$, else $y_i$ would belong to cluster $A_{i+1}$. Since $c_{i-1}$ is the average of the $y_j$'s belonging to $A_i-\{y_i\}$, the following inequality holds
$$\sum_{j \in A_i-\{y_i\}}(y_j-c_{i-1})^2<\sum_{j \in A_i-\{y_i\}}(y_j-u_i)^2$$
so that
\begin{align*}
&\sum_{j \in A_i-\{y_i\}}(y_j-c_{i-1})^2+(y_i-y_i)^2+\sum_{j \in A_{i+1}}(y_j-u_{i+1})^2\\ 
&\quad  < \sum_{j \in A_i-\{y_i\}}(y_j-u_i)^2+(y_i-u_i)^2+\sum_{j \in A_{i+1}}(y_j-u_{i+1})^2\\
 &\qquad  < \sum_{j \in A_i}(y_j-u_i)^2+\sum_{j \in A_{i+1}}(y_j-u_{i+1})^2.
\end{align*} 
This latter inequality indicates that the isotone regression with the values $\{c_{i-1},y_i,u_{i+1}\}$ would be better than the original one with $\{u_i,u_{i+1}\}$, which is in contradiction with the very definition of the isotone regression, therefore $y_i\leq u_i$.

\medskip

We prove in the same way that: 
$$y_{i+1}\geq u_{i+1}.$$
Hence, we obtain
$$ r_i=y_i-u_i  \leq 0 \qquad\textrm{and}\qquad r_{i+1}=y_{i+1}-u_{i+1}\geq 0$$ 
so that $ r_i\leq r_{i+1}$, and $b_i=b_{i+1}$, which is the desired result. 
\end{proof}

This result enables us to prove the so-called identifiability conditions for algorithms \textsf{I.I.R.} and  \textsf{I.I.B.R}.

\begin{corollary}\label{cor:identifiabilite}
For all $k\geq 1$,
\begin{equation}\label{eq:bir-ident1}
\Delta(\tilde{u}^{(k)})\circ \Delta(\tilde{b}^{(k)})=0.
\end{equation}
and
\begin{equation}\label{eq:bir-ident2}
\Delta(\hat{u}^{(k)})\circ \Delta(\hat{b}^{(k)})=0.
\end{equation}
\end{corollary}

\noindent\begin{proof}
Equation (\ref{eq:bir-ident1}) follows immediately from Proposition \ref{pro:intermediaire-identifiabilite}. On the other hand, equation (\ref{eq:bir-ident2}) requires more attention.  We prove this by induction. This is obviously true for $k=1$ by Proposition \ref{pro:intermediaire-identifiabilite}.  Let us fix $k\geq 1$, and assume that
$$\Delta(\hat{u}^{(k)})\circ \Delta(\hat{b}^{(k)})=0.$$
Since $\hat{u}^{(k+1)}=\hat{u}^{(k)}+\tilde{u}^{(k+1)}$ and $\hat{b}^{(k+1)}=\hat{b}^{(k)}+\tilde{b}^{(k+1)}$, we get
\begin{equation}\label{eq:rec1}
\Delta(\hat{u}^{(k+1)})\circ \Delta(\hat{b}^{(k+1)})=\Delta(\hat{u}^{(k)})\circ \Delta(\tilde{b}^{(k+1)})+ \Delta(\tilde{u}^{(k+1)})\circ \Delta(\hat{b}^{(k)}),
\end{equation} 
and our objective is to prove that both terms on the right-hand side of this equation are equal to zero. First notice that 
$$\tilde{b}^{(k+1)}=\anti(y-\hat{y}^{(k)}-\tilde{u}^{(k+1)})=\anti(y-\hat{b}^{(k)}-\hat{u}^{(k+1)})$$
and by definition of $\hat{u}^{(k+1)}$,
$$\hat{u}^{(k+1)}=\iso(y-\hat{b}^{(k)}).$$
Then, Proposition \ref{pro:intermediaire-identifiabilite} gives
$$ \Delta(\hat{u}^{(k+1)})\circ \Delta(\tilde{b}^{(k+1)})=0$$
so that 
$$\left\{\Delta(\hat{u}^{(k)})+\Delta(\tilde{u}^{(k+1)})\right\}\circ\Delta(\tilde{b}^{(k+1)})=0$$
and finally
$$\Delta(\hat{u}^{(k)})\circ\Delta(\tilde{b}^{(k+1)})=0.$$
Let us now turn to the second term on the right-hand side of equation (\ref{eq:rec1}). Just note that
$$\tilde{u}^{(k+1)}=\iso(y-\hat{y}^{(k)})=\iso(y-\hat{u}^{(k)}-\hat{b}^{(k)})$$
where
$$ \hat{b}^{(k)}=\anti(y-\hat{u}^{(k)}).$$
Invoking Proposition \ref{pro:intermediaire-identifiabilite} again, we are led to
$$ \Delta(\tilde{u}^{(k+1)})\circ \Delta(\hat{b}^{(k)})=0$$
and the right-hand side of equation (\ref{eq:pro1}) is equal to zero.
\end{proof}

\noindent{\bf Remark:} One can notice that the proof of Corollary \ref{cor:identifiabilite} uses the fact that algorithms \textsf{I.I.R.} and \textsf{I.I.B.R.} coincide (Theorem \ref{theo:egalite-algos}), but of course the proof of this theorem did not make any use of Corollary \ref{cor:identifiabilite}. A quick inspection shows that this last result is applied only in the demonstration of Corollary \ref{theo:iir-interpolation2}.

\end{document}